\documentclass[12pt,a4paper,twoside]{article}
\usepackage[utf8]{inputenc}
\usepackage[T1]{fontenc}
\usepackage{amsmath,amsfonts,amssymb}
\usepackage{mathptmx}  
\usepackage[top=3cm,bottom=3cm,left=3cm,right=3cm]{geometry}
\usepackage{xcolor}  
\usepackage{stmaryrd} \usepackage{hyperref}
\usepackage{amssymb}
\usepackage{amsmath} 
\usepackage{amsthm}  
\usepackage{newtxtext}
\hypersetup{
    colorlinks=true,
    linkcolor=blue,
}

\newtheorem{theorem}{Theorem}[section]
\newtheorem{lemma}[theorem]{Lemma}

\usepackage{fancyhdr}
\title{\huge{Factor Complexity of the Most Significant Digits of~$a^{n^d}$}}

\author{
  Mehdi Golafshan \thanks{Department of Mathematics, University of Liège,  Liège, Belgium.
\texttt{m.golafshan@uliege.be}} 
  \and 
  Ivan Mitrofanov\thanks{ Department of Mathematics, Saarland University, Saarbrücken, Germany. \texttt{phortim@yandex.ru}}
}
\date{}

\begin{document}

\maketitle

\begin{quote}
\small{
\textbf{Abstract.}
We investigate unipotent dynamics on a torus and apply these techniques to the following problem. Let \(d\) be a positive integer, and let \(a > 0\) be a real number. For an integer \(b \geqslant 5\), such that \(a\) and \(b\) are multiplicatively independent, consider the sequence \((\mathbf{w}_n)\), where \(\mathbf{w}_n\) is the most significant digit of \(a^{n^d}\) when expressed in base \(b\). We prove that the complexity function of the sequence \((\mathbf{w}_n)\) is, up to finitely many exceptions, a polynomial function.}
\end{quote}

\noindent \makebox[\textwidth][c]{\rule{.9\textwidth}{2pt}}
\begin{quote}
\quad
\noindent \textbf{Keywords:} 
The most significant digit 
 $\cdot$ Factor complexity $\cdot$
Unipotent dynamics on a torus.
$\cdot$
Symbolic dynamics
$\cdot$
Uniform distribution modulo $1$
$\cdot$
Equidistributed sequences
$\cdot$
Combinatorics on words
\end{quote}

\begin{quote}
\quad
\noindent \textbf{MSC 2020:}

\hspace{2cm}
\colorbox{yellow}{{\bf  Primary:}}
11K31; 37B10; 11A63

\hspace{2cm}
\colorbox{cyan}{{\bf  Secondary:}}
68R15; 11K06; 11K16
\end{quote}
\noindent \makebox[\textwidth][c]{\rule{0.8\textwidth}{2pt}}

\bigskip

\vspace{5mm}

\Large
\boxed{
\textbf{Funding}}
\normalsize

\medskip
\begin{quote}
Mehdi Golafshan is supported by the FNRS Research grant (\textbf{PDR 40014353}),
and the work of Ivan Mitrofanov
was supported by ERC Grant (grant agreement \textbf{No. 101097307}).
\end{quote}

\newpage
\tableofcontents  

\large
\section{Introduction}
The first recorded instance of the logarithmic distribution of leading digits was made in $1881$ by American astronomer Newcomb, who observed “how much faster the first pages wear out than the last ones'' and calculated the probabilities for the first and second digits using a few quick heuristics \cite{22}.

\medskip
For any non-zero real number \( x \) and any integer base \( b \geqslant 2 \), the {\it most significant digit} \( D_b(x) \) in base \( b \) is the unique integer \( j \in \llbracket 1, b-1 \rrbracket \) such that
\[
b^k \cdot j \leqslant |x| < b^k \cdot (j + 1)
\]
for some unique \( k \in \mathbb{Z} \).
 \cite{23}.

\medskip

The significant digits in many naturally occurring tables of numerical data do not follow an expected uniform distribution but rather, Benford's law, also known as   \textit{the significant-digit law}, is the empirical gem of statistical folklore.
More precisely,
in \cite{24},
Benford gave the exact law for the most significant digit is
$\log{\left( \frac{a+1}{a}\right)}$
for all $a \in \llbracket 1, 9\rrbracket$.
From a mathematical point of view, Benford’s Law is
 closely connected with the theory of uniform distribution modulo $1$ \cite{25}.
Using this relationship, Diaconis \cite{26} rigorously demonstrated in $1977$ that Benford's Law holds for a class of exponentially expanding sequences, including the sequence of factorials and powers of $2$, among others.

\section{Concepts and notations}

In this paper,
 the set of non-negative
integers (respectively integers, rational numbers, and real
numbers) is denoted by $\mathbb{N}$ (respectively, $\mathbb{Z}$, $\mathbb{Q}$, and $\mathbb{R}$).
Let $i,j \in \mathbb{Z}$ with
$i \leqslant j$. We let $\llbracket i,j\rrbracket$ denote the set of integers $\{i,i + 1,\ldots,j\}$.

\medskip

Two positive real numbers \(a\) and \(b\) are {\it multiplicatively independent} if \(\log_b a\) is not a rational number. In other words, there do not exist integers \(m\) and \(n\), not both zero, such that
\(
a^m = b^n
\)

\medskip

We make use of common notions in combinatorics on words, such as alphabet,
letter, word, length of a word, complexity function and usual definitions from symbolic dynamics.

Let $\mathcal{A}$ be a nonempty finite set of symbols, which we call an {\it alphabet}.
An
element $a \in \mathcal{A}$ is called a { \it letter}. A {\it word} over the alphabet $\mathcal{A}$ is a finite sequence
of elements of $\mathcal{A}$.
 Let $\epsilon$ denote the { \it empty word}. For a finite word $\textsc{\textbf{u}}$, we let $|\textsc{\textbf{u}}|$ denote
its length. For each $i \in \llbracket 0, |\textsc{\textbf{u}}|- 1 \rrbracket$, 
we let $\textsc{\textbf{u}}_i$ or $\textsc{\textbf{u}}[i]$ denote the $i^{\text{th}}$ letter of $\textsc{\textbf{u}}$ (and
we thus start indexing letters at $0$).

We use the notation $\mathcal{A}^*$ for the set of all
finite  words, and $\mathcal{A}^+=\mathcal{A}^* \setminus \{\epsilon\}$.
A word $\textsc{\textbf{u}}$ is a {\it factor}  of a word $\textsc{\textbf{w}}$ if there exist words
$\textsc{\textbf{p}}, \textsc{\textbf{q}} \in \mathcal{A}^*$
such that $\textsc{\textbf{w}} = \textsc{\textbf{p}}\textsc{\textbf{u}}\textsc{\textbf{q}}$.
In addition, $\textsc{\textbf{p}}$ is a {\it prefix} and $\textsc{\textbf{q}}$ is a { \it suffix} of $\textsc{\textbf{u}}$.

\medskip

Note that in the following, to avoid any ambiguity, the symbol $\#$ is used to denote the number of elements in a set, while the symbol $| \cdot |$ is used to refer to the number of letters forming a word, and in other places, it represents the absolute value.

This work investigates leading digit sequences from a complexity perspective.
There are numerous approaches to quantify the {\it complexity} of a word over a finite alphabet;
for instance, see \cite{5,6},
and
\cite{7}
for an overview of various complexity measures.
The factor complexity function will be used as our primary measure of complexity.

\subsection*{Complexity function}{\label{FC}}

The  {\it factor complexity} or  {\it complexity function} of a finite or infinite word  $\textsc{\textbf{w}}$ is the function
$n \mapsto \mathtt{p}_{\textsc{\textbf{w}}}(n)$,
which, for each integer $n$, gives the number $\mathtt{p}_{\textsc{\textbf{w}}}(n)$\footnote{If we do not need to refer to different words, we use the notation $\textsc{p}(n)$ for convenience.} of distinct factors  of length $n$ in that word \cite{8}.

\medskip

This  function  $ \mathtt{p}_{\textsc{\textbf{w}}}(n)$ has been extensively investigated.
 Morse and Hedlund in 
 $1938$
proved that an infinite word $\textsc{\textbf{w}}$ over a finite alphabet is eventually periodic if and only if 
 $\mathtt{p}_{\textsc{\textbf{w}}}(n)< n+1$
for some $n$
\cite{9}.

\medskip

\section{Summary of the result}

The following Olympiad problem was designed by  Kanel-Belov \cite{27}.

\medskip

\begin{quote}
    {\it  Consider the sequence 
$\textsc{\textbf{w}}$ 
whose $n^{\text{th}}$ term is the most significant  digit of $2^n$ in base 10.
Show that the number of different 
factors
made of $13$ consecutive digits 
from this sequence
is $57$,
i.e.,
$\mathtt{p}_{\textsc{\textbf{w}}}(13)=57$.
}
\end{quote}

\medskip

It can be shown that the number of different words with length $n$ is an affine function in $n$. Moreover, in \cite{28} it is shown that for all 
square-free bases $b \geqslant 5$ 
and all $a \in \mathbb{R}_{>0}$, where $a$ is not an integer 
power of $b$, the number of different factor of length $n$ in the sequence of the most significant digit of $a^k$ in base-$b$ representation is an affine function.

Separately in \cite{29}, this was is proved for some primes with a different method.

\medskip

In \cite{28}, it was suggested that the sequence of  the most  significant digit of $2^{n^2}$ has polynomial factor complexity. 
From Theorem \ref{main} it follows that the complexity function this sequence  is a polynomial of degree $3$
up to finitely many values.

\medskip

In this paper we prove the following result.

\begin{theorem}\label{main}
Let \(d \in \mathbb{Z}_{>0}\), let  \(b \geqslant 5\) be an integer and let \(a > 0\) be a real number such that \(a\) and \(b\) are multiplicatively independent. Consider the sequence \(\mathbf{w}\), where \(\mathbf{w}_n\) is the most significant digit of the number \(a^{n^d}\) when expressed in base \(b\).

Then there exists a polynomial \(\textsc{P}(k)\) of degree \(d(d+1)/2\) such that \(\textsc{P}(k) = \mathtt{p}_{\mathbf{w}}(k)\) for all sufficiently large \(k\).
\end{theorem}

\medskip

In \cite{30,31}, a similar following result is proved.
 
\medskip

\begin{theorem}
Let \(\textsc{Q}(n)\) be a polynomial with an irrational leading coefficient. Let \(\mathbf{w}_n\) be the first digit after the decimal point in the binary expansion of \(\textsc{Q}(n)\). Then there exists a polynomial \(\textsc{P}(k)\), depending only on \(\deg(\textsc{P})\), such that \(\textsc{P}(k) = \mathtt{p}_{\mathbf{w}}(k)\) for all sufficiently large \(k\).
\end{theorem}

\textbf{
Strategy:}
   moving to the unipotent dynamics on a torus and counting the number of regions into which the torus is divided by several families of parallel hyperplanes.
An additional complication of our problem is that the hyperplanes are not in a general position.

 \section{Proof of theorem~\ref{main}}
\label{proof_main}

\subsection{Equidistributed sequences}
A sequence of real numbers is said to be {\it equidistributed} or {\it uniformly distributed} if the proportion of terms that fall into any sub-interval is proportional to the length of this interval \cite{32},
i.e.,
a sequence $(s_i)_{i \in{\mathbb{N}_{>0}}}$ of real numbers is said to be equidistributed on a non-degenerate interval $[a,b]$ if for any sub-interval $[c,d]$ of $[a,b]$ we have
\begin{align*}
    \lim_{n\to \infty} \frac{\# \left(\{s_1, ..., s_n\} \cap [c,d]\right)  }{n} = \frac{d-c}{b-a}.
\end{align*}

\medskip

A sequence $(s_i)_{i \in{\mathbb{N}_{>0}}}$ of real numbers is said to be \textit{equidistributed modulo $1$} or \textit{uniformly distributed modulo $1$} if the sequence of fractional parts of $a_n$  is equidistributed in the interval $[0, 1]$
(see for instance \cite{25}).

\medskip

Let $X$ be a topological space with measure 
$\mu$.
We call a sequence $(a_n)_{n \in{\mathbb{N}_{>0}}} \in X$
{\it equidistributed}
if the proportion of terms that fall into any open set
$U \subset X$ tends to 
$\frac{\mu(U)}{\mu(X)}$
 \cite{52}.

Let $\mathbb{T}^d:=[0,1)^d$ be the torus $\mathbb{R}^d /\mathbb{Z}^d$.
Denote by $\rho(x,y)$ both the Euclidean metric in $\mathbb{R}^d$ and the locally Euclidean metric in $\mathbb{T}^d$, and we consider it with the Lebesgue measure.

We now proceed by introducing Weyl's theorem.

\begin{theorem}\cite{33}\label{thwe}
	Let  $\textsc{P}(t)$ be a be a polynomial
	with at least one irrational coefficient. 
	Then the sequence of fractional parts 
	$(\{\textsc{P}(i)\})_{i \in \mathbb{N}_{>0}}$
	is equidistributed (and in particular dense) in $\mathbb{T}^1$.
	
\end{theorem}

\medskip

\medskip

\begin{theorem}\cite{34}\label{thci}
	Let  $\textsc{P}(t) = a_0t^d + \cdots + a_d$  be a polynomial
	with real coefficients, where $a_0$ is irrational. 
	Then the sequence of $d$-tuples 
	$$
	\left(\{\textsc{P}(n)\}, \{\textsc{P}(n+1)\}, \ldots, \{\textsc{P}(n+d-1)\}\right)
	$$
	is equidistributed in $\mathbb{T}^d$.
\end{theorem}


\subsection{Unipotent dynamics on a torus.}

The condition that the most significant  digit $\textsc{\textbf{w}}_n$ of the number $ a ^ {n ^ d} $ in base $b$ is to $ t $ can be written as
\begin{align*}
    \log_b(t) \leqslant \{\log_b(a)  \cdot n^d\} < \log_b{(t+1)}. 
\end{align*}

We denote $\log_b(a)$ by $\zeta$, and know that $\zeta \not\in \mathbb{Q}$.

\medskip

Consider a $d $-dimensional torus $\mathbb{T}^d$, and match the number $n$ to a point $v_n\in\mathbb{T}^d$ with coordinates
$$
\left(\{\zeta n\}, \{\zeta n^2\}, \cdots, \{\zeta n^d\}\right).
$$

It is easy to show that if $v_n$ has coordinates $(v_n^{(1)},\cdots, v_n^{(d)})$, then $v_{n+1}$ has coordinates 
$$	
\begin{array}{rcl} v_{n+1}^{(1)} & = & 
\{v_n^{(1)} + \zeta\}, \\
v_{n+1}^{(2)} & = & \{v_n^{(2)} + 2v_n^{(1)} +\zeta \}, \\
& \vdots &  \\
v_{n+1}^{(d)} & =& \{v_n^{(d)} + dv_n^{(d-1)}+\cdots + \binom{d}{k} v_n^{(d-k)} + \cdots   + d v_n^{(1)} + \zeta\}.\\
\end{array} 
$$

\medskip

It is straightforward to verify  that the map $f: \mathbb{T}^d \to \mathbb{T}^d$ defined by these formulas is a self-homeomorphism of $\mathbb{T}^d$.
Let us denote the corresponding affine transformation by $A:\mathbb{R}^d\to \mathbb{R}^d$.

Let us define {\it critical } subset of $S_t$ of $\mathbb{T}^d$ for all
$ t  \in  \llbracket 1,b-1\rrbracket$
as 
\begin{align*}
    S_t:=\{x \in \mathbb{T}^d: \, x^{(d)}=\log_b(t)\}.
\end{align*}
The critical sets divide $\mathbb{T}^d$ into $b-1$ regions $U_1,\cdots,U_{b-1}$; the set $U_t$  for all $ t  \in  \llbracket 1,b-1\rrbracket$ is defined as follow
\begin{align*}
    U_t:=\{ x \in \mathbb{T}^d: \, 
    \log_b(t)\leqslant
    x^{(d)}< \log_b(t+1)\}.
\end{align*}
Then the digit $ \textsc{\textbf{w}}_k$ is  there the index of $U_t$ where the point $v_k =  f^k (0, \cdots, 0) $ is located.

\medskip

We denote  the set $f^{-k} (S_t)$ by $S_{t,k}$.
The sets $S_{t,k}$ with $ t  \in  \llbracket 1,b-1\rrbracket$ and $ k  \in   \llbracket 1,n\rrbracket$ divide $\mathbb{T}^d$ into connected regions; the  set of these regions is denoted by $\mathcal{M}_n$.

Let $ \textsc{\textbf{u}} = u_1 \cdots u_n $ be a sequence of digits such that $|\textsc{\textbf{u}} |=n$, then we
assign
\begin{align*}
    U_{\textsc{\textbf{u}} }:=
U_{u_1} \cap f^{-1}(U_{u_2})\cap \cdots \cap f^{-(n-1)}(U_{u_n}).
\end{align*}

The set of all  $U_{\textsc{\textbf{u}} }$ where 
$|\textbf{u}|=n$ is denoted by $\mathcal{U}_n$.
It is
clear that $\mathcal{M}_n$ and $\mathcal{U}_n$ are partitions of $\mathbb{T}^d$, and $\mathcal{M}_n$ is a sub-partition of $\mathcal{U}_n$.

\newpage

\begin{lemma}
The sequence $(v_i)_{ i \in \mathbb{N}_{>0}}$ is dense in $\mathbb{T}^d$.
\end{lemma}

\begin{proof}

It is not difficult to show that the monomials 
$x,x^2,\cdots, x^d$ are rational linear combinations of the polynomials
$$x^d, (x+1)^d,\cdots, (x+k-1)^d, \ \text{and} \  1.$$

Let $N$ be the common denominator of the coefficients of these combinations.
We apply Theorem \ref{thci} to the following polynomial to derive the lemma:
\[
P(x) = \frac{\zeta}{N} x^d.
\]
\end{proof}

\begin{lemma}
A finite sequence of digits $\textsc{\textbf{u}}$ is a factor of $\textsc{\textbf{w}}$ if and only if the set $U_{\textsc{\textbf{u}} }\subset \mathbb{T}^d$ has  a non-empty interior.
\end{lemma}

\begin{proof}
 Definition of $U_{\textsc{\textbf{u}} }$  implies that  $\textsc{\textbf{u}}$ is a factor of $\textsc{\textbf{w}}$ if and only if
$v_k \in U_{\textsc{\textbf{u}} }$ for some $k$.
Since the sequence $(v_k)_{ k\in \mathbb{N}_{>0}}$ is dense in $\mathbb{T}^d$, any subset of $\mathbb{T}^d$ with non-empty interior contains an element of this sequence. 

On the other hand, let $\textsc{\textbf{u}} = u_0\cdots u_n$ be a prefix of $\textsc{\textbf{w}}$, 
afterwards $f^k(0)\in U_{u_k}$ for all $  k  \in  \llbracket 0,n\rrbracket$. 

If $\Delta\in \mathbb{R}_{>0}^d$ is a vector with small enough positive coefficients, then $f(x+\Delta) - f(x)$ is  a small vector with positive coordinates, too (here we interpret $\mathbb{T}^d$ as an abelian group).
Hence, $U_{\textsc{\textbf{u}} }$ contains a small open ball.
Now let $\textsc{\textbf{u}}  = u_1u_2$. Since
\begin{align*}
    f^{|u_1|}(U_{u_1u_2})\subset U_{u_2},
\end{align*}
$U_{u_2}$
has non-empty interior, too.
\end{proof}

\medskip

Let $x,y\in \mathbb{T}^d$. We define an equivalence relation  $x\sim_ky$ if for any
$  i  \in  \llbracket 0,k-1\rrbracket$ the points $f^k(x)$ and $f^k(y)$ belong to the closure of the same region $U_t$ (i.e., $x$ and $y$ belong to the closure of the same $U$ from $\mathcal{U}_k$). 
We say that
 $x\sim y$ whenever $x\sim_ky$ for all $k\in \mathbb{N}_{>0}$.
If $x,y\in \mathbb{R}^d$, we say that $x\sim_ky$ if $x'\sim_ky'$ where $x'$ and $y'$ are projections under the factor mapping 
\begin{align*}
    \pi:\mathbb{R}^d &\to \mathbb{T}^d,\\ \pi(x^{(1)},\cdots, x^{(d)}) & = \left(\{x^{(1)}\},\cdots, \{x^{(d)}\}\right).
\end{align*}

For instance, if $d = 2$, then for any $(x,y)\in \mathbb{T}^2$
we have $$(x,y)\sim \left(x+\frac{1}{2}, y\right).$$

Let $x,y \in \mathbb{R}^d$. 
It is easy to see that $A(y) - A(x) = M(y - x)$,
where

\[
M = 
\begin{bmatrix} 
1 & 0 & 0 &  0 & \cdots & 0 \\
2 & 1 & 0 & 0 & \cdots& 0 \\
3 & 3 & 1 & 0 & \cdots& 0 \\
\vdots &  & &\ddots & & \vdots\\
\vdots &  & & & \ddots & \vdots\\
d & \binom{d}{2}  & \dots &  \binom{d}{2} & d & 1 \\
\end{bmatrix}.
\]

We call a vector $r\in \mathbb{R}^d$ {\it stable} when for any $k\in \mathbb{N}_{>0}$ the last coordinate of $M^k(r)$ is integer. 

\begin{lemma}
The set of all stable vectors forms a discrete subgroup in  $\mathbb{R}^d$.
\end{lemma}

\begin{proof}
It is obvious that stable vectors form a subgroup in $\mathbb{R}^d$, accordingly we just need to prove that it is discrete.
Let $r = (r_1,\cdots, r_d)$.
It is not difficult to show by induction on $i$ that the $i$-coordinate 
\begin{align*}
    \left(M^k(r)\right)^{(i)} = r_1\textsc{P}_{i,1}(k) + \cdots + r_{i-1}\textsc{P}_{i,i-1}(k) + r_i,
\end{align*}
where $\textsc{P}_{i,j}$ is a polynomial of degree $(i-j)$ that does not depend on $r$.

\medskip

In particular,
\begin{align*}
    \left(M^k(r)\right)^{(d)} = r_1\textsc{P}_{d,1}(k) + \cdots + r_{d-1}\textsc{P}_{d,d-1}(k) + r_d.
\end{align*}

Let $C_d\in \mathbb{R}$ be such that $\textsc{P}_{d,j}(x) \leqslant C_d$ for all $j \in  \llbracket 1, d\rrbracket  $ and $x \in \llbracket 0, d+1\rrbracket$.
Suppose that for all $i$,
\begin{align*}
    |r_i|< \frac{1}{dC_d},
\end{align*}
and that $\left(M^k(r)\right)^{(d)} \in \mathbb{Z}$ for $  k  \in  \llbracket 1,d+1\rrbracket$.
Since 
\begin{align*}
  \left|\left(M^k(r)\right)^{(d)}\right| < \frac{d C_d}{d C_d} = 1,  
\end{align*}
 and $\left(M^k(r)\right)^{(d)} = 0$ for $  k  \in  \llbracket 1,d+1\rrbracket$.
Yet $\left(M^k(r)\right)^{(d)}$ is a polynomial in $k$ of degree $d$.
We conclude that 
\begin{align*}
    r_1 = r_2 = \cdots = r_d = 0.
\end{align*}
\end{proof}

\medskip

It is clear that for any stable vector $r$ it holds $v_1 \sim (v_1 + r)$.
We can also consider stable vectors as a finite subgroup in $\mathbb{T}^d$.

\medskip

\begin{lemma}\label{lestable}
Let $x,y\in \mathbb{T}^d$ be such that $y-x$ is not stable. 
Then $v_1 \not \sim v_2$.
\end{lemma}

\begin{proof}
Let $r := y-x$.
Consider two cases:

\begin{itemize}
    \item 
 {\bf Case I:}  The vector $r\in \mathbb{Q}^d$.
    Then $M^N(r) \equiv r  \pmod{\mathbb{T}^d} $ for some $N$, and we can consider that  $r^{(d)}\not \in \mathbb{Z}$.
It is clear that $A^k(v_1)^{(d)} = \textsc{P}(k)$ where $\textsc{P}(k)$ is a polynomial of degree $d$ with leading coefficient $\log_b(a)$. 
We apply Theorem \ref{thwe} for the polynomial $\textsc{P}(Nx)$, the set of fractional parts $\{A^{Nk}(x)^{(d)}\}$ is dense in $\mathbb{T}^1$. 
Choose $k$ so that $f^{Nk}(x)$ and $f^{Nk}(x)+r$ are in different regions $U_t$.

\medskip

\item
{\bf Case II:} Not all $r_i\in \mathbb{Q}$.
 Suppose that $i$ is the smallest index such that $r^{(i)}\not\in \mathbb{Q}$. 
It can be shown that $M^k(r)^{(d)}$ is a polynomial in $k$ with an irrational coefficient of $k^{d-i}$. 
According to
 Theorem \ref{thwe}, we claim that   $\{M^k(r)^{(d)}\}$ can be arbitrarily close to $1/2$; 
in this case $f^{k}(x)$ and $f^k(y)$ are in different regions $U_t$, since each $U_i$ has the length of projection to the last coordinate shorter than 
$1/2 $ (maximal length $\log_b(2)$ for $U_1$).

\end{itemize}
\end{proof}

Stable vectors form a subgroup $\mathfrak{R}$ in $\mathbb{T}^d$.
 Lemma \ref{lestable} implies that each $\sim$ equivalence class of $\mathbb{T}^d$ contains exactly $\#\mathfrak{R}$ points.

\medskip

Denote by $\rho_{\mathfrak{R}}(\cdot,\cdot)$ the pseudo-metric in $\mathbb{T}^d$ defined by
$$
\rho_{\mathfrak{R}}(x,y):= \min_{r\in {\mathfrak{R}}} \rho(x,y-r).
$$ 

\medskip

\newpage

\begin{lemma}\label{lecomp}
For any $\delta>0$ there exists $N(\delta)$ such that
for any $x,y\in \mathbb{T}^d$ with $\rho_{\mathfrak{R}}(x,y) > \delta$:

\begin{enumerate}
	\item[\texttt{(i)}.] for some $0 \leqslant k_1 < N(\delta)$,  $f^{k_1}(x) \not \sim_1 f^{k_1}(y)$ (equivalently, $x\not \sim_k y$);
	\item[\texttt{(ii)}.] for some $0 \leqslant k_2 < N(\delta)$,  $f^{-k_2}(x) \not \sim_1 f^{-k_2}(y)$.
\end{enumerate} 
\end{lemma}

\begin{proof}
\texttt{(i)}:
Assume the contrary.
All the relations $\sim_k$ define closed subsets of $\mathbb{T}^d\times \mathbb{T}^d$. 
By compactness argument we find two different points $x$
and $y$ such that $x \sim y$, yet $\rho(x,y) \geqslant \delta$, 
a
contradiction with Lemma \ref{lestable}. 

\medskip

\texttt{(ii)}:
To prove this note that all previous lemmas hold whenever we consider
$f^{-1}$ instead of $f$ if $f$ invertible.
\end{proof}

\begin{lemma}
There exists $N$ such that for all $n> N$, $\mathcal{M}_n$ is a partition of $\mathbb{T}^d$ into convex polyhedrons and every $U$ from $\mathcal{U}_n$ consists of exactly  $\#\mathfrak{R}$ regions from $\mathcal{M}_n$.
\end{lemma}

\begin{proof}
Let
$K$
be the length of the shortest stable vector.
Assume that
$M$
is a
$\lambda$-Lipschitz
transformation of the torus.
Suppose that
 $\delta < {K}/{4\lambda}$ and that $N = 3 \cdot N(\delta)$ (the notation used in  Lemma \ref{lecomp}).

Let $\textsc{\textbf{u}} $ be a word of length $n \geqslant N$ and  the interior of $U_{\textsc{\textbf{u}} }$ be non-empty.
Let $x\in U_{\textsc{\textbf{u}} }$, and let $x=x_1 x_2 \cdots x_{\#{\mathfrak{R}}}$ be all shifts of $x$ by a vector from ${\mathfrak{R}}$. 
Obviously, all $x_i\in U_{\textsc{\textbf{u}} }$. 
By definition of $N(\delta)$, we can see that $\rho_{\mathfrak{R}}(x,y) \leqslant \delta$ for any $y\in U_{\textsc{\textbf{u}} }$.
It means that $U_{\textsc{\textbf{u}} }$ belongs to the union of ${\#\mathfrak{R}}$ balls of radius $\delta$ each centered at $x_i$. 

We will show that intersection of $U_{\textsc{\textbf{u}} }$ with any of these balls consists of one region from $\mathcal{M}_n$ (and since this region is bordered by sets $S_{t,i}$ that are locally hyperplanes, this part is a convex polyhedron).

Assume that the interior of $U_{\textsc{\textbf{u}} }$ has two points $y,z$ such that the segment of length $<\delta$ connects $y$ and $z$ and intersects a set $S_{t,k}$ for some $t$ and $k \leqslant n$.

Let $\textsc{\textbf{u}}[k]=\beta$.
We will consider $x$ and $z$ as points of $\mathbb{R}^d$,
$\rho(x,y) < \delta$.
Denote 
\begin{align*}
   q_y:=A^k(y)^{(d)} \ \text{and} \ q_z:=A^k(z)^{(d)}.
\end{align*}
Since $\textsc{\textbf{u}}[k]=\beta$, 
\begin{align*}
    \log_b(\beta) \leqslant \{q_y\}, \, \{q_z\} <\log_b(\beta+1).
\end{align*}

But the segment that connects $A^k(y)$ and $A^k(z)$ contains a point with $d$-coordinate equal to $\log_b{(t)}$, hence $\lfloor q_x\rfloor \neq \lfloor q_y\rfloor $ and 
\begin{align*}
    \rho\left(A^k(x),A^k(y)\right) > \frac{1}{2}.
\end{align*}

Since the transformation $A$ is $\lambda$-Lipschitz,  
\begin{align*}
    \frac{K}{2\lambda} \leqslant \rho\left(A^i(y), A^i(z)\right) < \frac{K}{2}
\end{align*}
for some $1 < i < n$.
Then 
\begin{align*}
    \rho_{\mathfrak{R}}\left(f^i(y),f^i(z)\right) = \rho\left(A^i(y), A^i(z)\right) \geqslant \frac{K}{2\lambda} > \delta.
\end{align*}

There are two similar cases:

\begin{itemize}
    \item {\bf Case I:}
$i \leqslant n/2$.  
    Similarly violates  statement (2)
    of Lemma \ref{lecomp}.

    \item
    {\bf Case II:} 
 $i > n/2$. Since $y\sim_n z$, thus
$f^{i}(y) \sim_{(n-i)} f^i(z)$, but  this contradicts to Lemma \ref{lecomp} because $(n - i) > N(\delta)$.

\end{itemize}
\end{proof}

So, $\mathtt{p}_{\textsc{\textbf{w}}}(k) = \#\mathcal{M}_k/{\#\mathfrak{R}}$ for $k > N$. 
Therefore, to prove Theorem\ref{main}, we need to count the number of such domains and show that it is a polynomial of a single indeterminate $k$.

\subsection{Counting the number of regions in a hyperplane torus partition}
\label{proof_begin}

\begin{lemma}
\label{le:pol}

Let \( d \) and \( \alpha \) be positive integers. 
Let \( \textsc{P}(x) \) be a monic polynomial of degree \( d \), and let \( x_0 < \cdots < x_d \) be integers such that \( |\textsc{P}(x_i)| \leqslant \alpha \) for each \(  i \in \llbracket 0, d\rrbracket  \). Then \( |x_d - x_0| \) is bounded from above by a constant \( C \) that depends only on \( d \) and \( \alpha \).

\end{lemma}

\begin{proof}
By Lagrange interpolation formula, the leading coefficient is
\begin{align*}
    1 = \sum_{i = 0}^d \frac{\textsc{P}(x_i)}{(x_i - x_0)\cdots  (x_i - x_{i-1})(x_{i}-x_{i+1})\cdots  (x_i - x_d)}.
\end{align*}

Since each term is less than and equal to  $ \alpha/(x_d - x_0-1)$, we conclude that
\begin{align*}
    |x_d-x_0| \leqslant \alpha(d+1)  + 1.
\end{align*}
\end{proof}

Consider the \( d \)-dimensional torus \( \mathbb{T}^d = [0, 1)^d \). Define the {\it critical subsets} of \( \mathbb{T}^d \) by the conditions:
\[
    x_d = 0, \quad x_d = \log_b(2), \quad \cdots, \quad x_d = \log_b(b-1),
\]
where \( x_d \) denotes the \( d \)-coordinate of \( x \in \mathbb{T}^d \), and \( b \geqslant 5 \) is a fixed integer.

We recall the notation $\zeta = \log_b(a)$ and the map \( f: \mathbb{T}^d \to \mathbb{T}^d \) defined by:
\[
\begin{aligned}
x_1' &= \{ x_1 + \zeta \}, \\
x_2' &= \{ x_2 + 2x_1 + \zeta \}, \\
x_3' &= \{ x_3 + 3x_2 + 3x_1 + \zeta \}, \\
&\vdots \\
x_d' &= \left\{ x_d + d x_{d-1} + \binom{d}{2} x_{d-2} + \dots + \binom{d}{d-1} x_1 + \zeta \right\},
\end{aligned}
\]
where \( \{ x \} \) denotes the fractional part of \( x \), and \( \binom{d}{k} \) is the binomial coefficient.

\begin{lemma}\label{le:intersect}
For each point \( x \in \mathbb{T}^d \), define the set \( I_x \subseteq \mathbb{Z} \) as:
\[
I_x = \{ k \in \mathbb{Z} : f^k(x) \text{ lies in one of the critical subsets} \}.
\]
 
\begin{enumerate}
    \item[\texttt{(i)}.]
     For every finite subset \( I \subset \mathbb{Z} \) of size \( d \), the set of points \( x \in \mathbb{T}^d \) such that \( I \subseteq I_x \) is finite.

  \item[\texttt{(ii)}.]
   There exists a constant \( C \) depending only on $d$, $a$ and $b$ such that if \( \# I_x > d \), then:
\[
\max_{k \in I_x}{(k)} - \min_{k \in I_x}{(k)} \leqslant C.
\]
\end{enumerate}

\end{lemma}

\begin{proof}
\texttt{(i)}.
	It is easy to verify it by induction on $k$ that if $x = (x_1,\ldots, x_d),$ then the point $f^k (x)$ has the $d$-coordinate
	\begin{align*}
	    f^k(x)^{(d)} = \left\{x_d + dkx_{d-1}+\cdots + \binom{d}{i} k^ix_{d-i} + \cdots +d k^{d-1} x_1 + k^d\zeta \right\}.
	\end{align*}
	
We write for $d$ different indices $k$ equalities of the form
\begin{align*}
    x_d + dkx_{d-1}+\cdots + \binom{d}{i}k^ix_{d-i} + \cdots +d k^{d-1} x_1 + k^d\zeta = \log_b(t_k) + n_k,
\end{align*}
where $t_k$ is an integer from $1$ to $b-1$, and $n_k$ is an arbitrary integer.

For
$d$
different
$k_1 < \cdots < k_d$
this result is a system of $d$ equations with $d$ unknown $x_1, \cdots, x_d $.
This system of linear equations has matrix
\begin{align*}
    M_{k_1,\dots,k_d} = 
\begin{bmatrix} 
1 & d k_1 & \dots &  \binom{d}{i} k_1^i & \dots & dk_1^{d-1} \\
\vdots & \ddots & & & & \vdots\\
\vdots &  & \ddots& & & \vdots\\
\vdots &  & &\ddots & & \vdots\\
\vdots &  & & & \ddots & \vdots\\
1 & d k_d & \dots &  \binom{d}{i} k_d^i & \dots & dk_d^{d-1} \\
\end{bmatrix}.
\end{align*}

\medskip

This matrix
is invertible,
because it is
 is
a Vandermonde matrix with integer coefficients multiplied by a non-degenerate 
integer diagonal matrix.
Hence each set of
$\{t_k\}$
and
$\{n_k\}$
corresponds to exactly
$1$ solution, and each component of the solution is a $\mathbb{Q}$-combination of numbers $\log_b(2),\dots, log_b(b-1),$ and $\zeta$. 
If
we fix the set
$\{t_k\}$
and consider all possible integer vectors
$(n_{k_1}, \cdots, n_{k_d})$,
the set of solutions will be a shift of the lattice
$M_{k_1,\dots,k_d}^{-1}(\mathbb{Z}^d)$.
Since
$M_{k_1,\dots,k_d}^{-1}$
has only rational coefficients, the number of solution inside torus $\mathbb{T}^d$
will be finite
( and equal to 
$\det{(M_{k_1, \ldots, k_d})}$),
so
\texttt{(i)}
is proved.

\medskip

\texttt{(ii)}.
Let $V$ be the vector space over $\mathbb{Q}$, generated by numbers $1,\zeta$ and $\log_b(i)$ for $i \in \llbracket 2, b-1 \rrbracket$. Let $v_1,\dots, v_m$ be a basis of this vector space. 
Let $1 = \alpha_1 v_1 + \dots + \alpha_k v_m$ and $\zeta = \beta_1 v_1 + \dots + \beta_m v_m$. 
Since $\zeta\not \in \mathbb{Q}$, we can assume that 
\[
D:=\det \begin{pmatrix}
\alpha_1 & \alpha_2 \\
\beta_1 & \beta_2
\end{pmatrix}  \neq 0.
\]

\medskip

If
$\#I_x \geqslant d$,
then we know that
the components of $x$, i.e.
each of the $x_i$'s are elements of $V$.
Let $x_i = \gamma_{i,1}v_1 + \dots + \gamma_{i,m} v_m$ for all $\gamma_{i,j}\in \mathbb{Q}$.

\medskip

If $k\in I_x$,  then 
\begin{align}\label{x_k}
    x_d + dkx_{d-1}+\cdots + \binom{d}{i} k^ix_{d-i} + \cdots +d k^{d-1} x_1 + k^d \zeta = 
    \log_b(t_k) +n_k
\end{align}

for some $t_k\in \llbracket 1,b-1 \rrbracket$ and $n_k\in \mathbb{Z}$.

\medskip

Let $\log_b(t_k) := \delta_{k,1}v_1+\dots+\delta_{k,m}v_m$ where all $\delta_{i,j}\in \mathbb{Q}$.

\newpage

Rewriting  \ref{x_k} in basis $v_1,\dots, v_m$ and looking at $v_1$ and $v_2$ components gives us two following relations:

\[
\begin{cases}
\gamma_{d,1} + \cdots + \binom{d}{i} k^i\gamma_{d-i,1} + \cdots +d k^{d-1} \gamma_{1,1} + k^d \beta_1 = 
    \delta_{k,1} +n_k \alpha_1 \\
\gamma_{d,2} + \cdots + \binom{d}{i} k^i\gamma_{d-i,2} + \cdots +d k^{d-1} \gamma_{1,2} + k^d \beta_2 = 
    \delta_{k,2} +n_k \alpha_2
\end{cases}
\]

For $i \in \llbracket 1,d \rrbracket$ denote $\alpha_1\gamma_{i,2} - \alpha_2\gamma_{i,1}$ by $r_i$. 
Denote $\alpha_1 \delta_{k,2} - \alpha_2 \delta_{k,1}$ by $c_k$.
Then

\begin{align*}
r_{d} + \cdots + \binom{d}{i} k^ir_{d-i} + \cdots +d k^{d-1} r_1 + k^d D = 
    c_k.
\end{align*}

Note that all $c_k$ belong to a finite set that depends only on $a$ and $b$.

Assume that
$\#I_x \geqslant d+1$ and 
apply Lemma \ref{le:pol} to the polynomial
\begin{align*}
    \textsc{P}(k) = k^d + \cdots + \binom{d}{i} \frac{r_{d-i}}{D} k^i + \cdots + \frac{r_d}{D}.
\end{align*}

We get that the difference between the largest and smallest elements of $ I_x $ is bounded from above by some constant that does not depend on $x$.

\end{proof}

\begin{lemma}
	Let $k$ be a natural number.
 The union of $S_{t,0},  \ldots, S_{t,k-1}$ divides $\mathbb{T}^d$ into $N(k)$ connected regions.
 Then
 $N(k)$ starting from some
 $k_0$ is a polynomial in
 $k$ and the degree of this polynomial is
 $d(d+1)/2$.
\end{lemma} 

\begin{proof}
Consider the surfaces
 $$S_{t,i}:=f^{-i}\left(\{x\in \mathbb{T}^d: x^{(d)} = \log_b(t)\}\right).$$

From
statement
\texttt{(i)}
of the previous Lemma, we see that the set of the points
in
$\mathbb{T}^d$,
which belong to at least $d$ of these surfaces, is finite.

Shifting each $S_{t,i}$ slightly parallel to itself so that the resulting surfaces intersect no more than $d$ at one point, we get a family of surfaces $\{S'_{t, i}\}$.
First of all, we show that the number $ N'(k) $ of domains into which $\mathbb{T}^d$ is partitioned by $ \{S'_{t, i} \} $ is a polynomial in $k$.

We assume that $k \geqslant d$.
Choose an arbitrary vector $v$ that is transversal to all $S_{t,i}$.
 Each domain $\Phi$ from the partition is a convex polyhedron; this polyhedron has a single vertex $v(\Phi)$ that minimizes the linear functional of the scalar product $\langle v, \cdot \rangle$. 
It is easy to see that this correspondence gives a bijection between the regions of the partition points that belong to $d$ of the surfaces $\{S'_{t,i}\}$.

Consider $d$ surfaces $S'_{t_1,i_1}, \ldots, S ' _{t_d,i_d}$.
It is clear that the number of $d$-intersection points depends only on the set $(i_1, \ldots, i_d)$, denote this number by $N(i_1,\ldots, i_d)$.
Set
\begin{align*}
    h_i(x_1,\ldots, x_d)= x_d + dix_{d-1}+\cdots + \binom{d}{j}i^jx_{d-j} + \cdots + d i^{d-1} x_1.
\end{align*}

The number $N(i_1,\ldots, i_d)$ is equal to the number of solutions of the system
$f_{i_l}(x)\in \mathbb{Z}$ 
for $l \in  \llbracket 1,d \rrbracket$
per unit volume, i.e., the absolute value of the determinant of the matrix
\begin{align*}
    M_{i_1,\dots,i_d} = 
\begin{bmatrix} 
1 & d i_1 & \dots &  \binom{d}{j} i_1^j & \dots & di_1^{d-1} \\
\vdots & \ddots & & & & \vdots\\
\vdots &  & \ddots& & & \vdots\\
\vdots &  & &\ddots & & \vdots\\
\vdots &  & & & \ddots & \vdots\\
1 & d i_d & \dots &  \binom{d}{j} i_d^j & \dots & di_d^{d-1} \\
\end{bmatrix}.
\end{align*}

It is easy to see via  Vandermonde determinant that
\begin{align*}
    N(i_1,\ldots, i_d) =&\left|\det{(M_{i_1,\cdots,i_d})}\right|\\ = &d\cdot \binom{d}{2} \cdots  \binom{d}{d-1} \cdot \prod_{1 \leqslant j_1 < j_2 \leqslant d} (i_{j_2} - i_{j_1}),
\end{align*}
and therefore, the total number of regions is equal
\begin{align*}
    N'(k) =  (b-1)^d \cdot d\cdot \binom{d}{2} \cdots \binom{d}{d-1} \cdot \sum_{0 \leqslant i_1 < \cdots < i_d \leqslant k} \ \ \prod_{1 \leqslant j_1 < j_2 \leqslant d} (i_{j_2} - i_{j_1}).
\end{align*}

It is easy to verify that  the right is a linear combination of a finite number (depending only on $d$) of sums of the form
\begin{align*}
    \sum_{0 \leqslant i_1 < \cdots < i_d \leqslant k}  p(i_{j_1},\ldots,i_{j_d}),
\end{align*}
where $p$ is a monomial.

By induction on $k$ and  $d$, it is easy to obtain that each such sum is a polynomial in $k$.

So $N'(k)$ is also a polynomial in $k$.
The function $N'(k)$ behaves asymptotically as $k^{d(d+1)/2}$, which means that it is a polynomial of degree equal to ${d(d+1)/2}$.

Now we find the difference between $N(k)$ and $N'(k)$.
We call a point of $\mathbb{T}^d$ where $d' > d $ of surfaces $S_{t,i}$ intersect {\it singular}.
Each singular point decreases the number of the regions by  $c_{d'}$, where $c_{d'}$ depends only on $d'$ (and is equal to the number of bounded regions of the partition of $\mathbb{R}^d$ by $d'$ hyperplanes of general position).  

It follows from Lemma \ref{le:intersect} that the numbers $d'$ are bounded by some $C(d)$.
Each singular point has form $f^k(x)$ where $x\in S_{t,1}$ is a singular point.
Thus the number of singular points for each $d'$ is an affine function in $k$ (for all $k$ greater than some $k_0$).

Hence, $N'(k)-N(k)$ is an affine function after some point, and hence the value $N(k)$ is after some point a polynomial of degree $d(d+1)/2$.
\end{proof}

\section{Concluding remarks}
In this section, we go over some relevant ideas and intriguing problems that our findings raise.

\subsection{Other complexity functions}
In Section \ref{FC} we defined the  complexity function of a word.
   J.  Cassaigne brings up an intriguing problem \cite{35}:

  \medskip
    
    \textbf{Problem 1.}
    \textit{Is it possible to find a family of word $\mathfrak{F}$ such that the complexity function be equal to a given function f(n) ?}

  \medskip
    
    Such families have been discovered for some especial functions.
   For instance,
$\mathtt{p}_{\textsc{\textbf{w}}}(n) = n+1$
    if and only if $\textsc{\textbf{w}}$ is Sturmian.
    Also
    $\mathtt{p}_{\textsc{\textbf{w}}}(n) = \mathcal{O}(1)$
    if and only if $\textsc{\textbf{w}}$ is eventually periodic;
    and
    $\mathtt{p}_{\textsc{\textbf{w}}}(n) = n+ \mathcal{O}(1)$
    if and only if $\textsc{\textbf{w}}$ is  quasi-Sturmian
    \cite{50}.

\medskip

In the general case of a coding of an irrational rotation, the complexity
has the form $\mathtt{p}(n) = an + b$, for $n$ large enough (see Theorem 10 in \cite{14,18}, for the whole proof).

The converse is not true: every word
of ultimately affine complexity is not necessarily obtained as a coding of
rotation. Didier gives in \cite{19} a characterization of codings of rotations.

On the other hand, A. Heinis covered some of this subject in his Ph.D. thesis \cite{36}.

In \cite{37},   the class $\mathtt{p}(n) = n+ {o}(n)$ is studied, i.e., $\lim_{n \to \infty} {\mathtt{p}(n)}/{n}=1$;
also   Rote  considered  words with complexity  $\mathtt{p}(n) =2n$ (see \cite{20}).
Recently, Cassaigne,  Labbé, and  Leroy introduced a set of words of complexity $2n+1$ in \cite{38}.
    Based on these considerations,  we have found a set of words whose function has polynomial behavior.
 However, if the complexity of a word $\textsc{\textbf{u}}$ has
the form $\mathtt{p}(n) = n + k$, for $n$ large enough, then $\textsc{\textbf{u}}$ is the image of a
Sturmian word by a morphism, up to a prefix of finite length
(see for instance \cite{18,21}).

\subsection{Types of complexity}
In  combinatorics of words, various measures of complexity have been introduced (due to space limitation, we only mention some bibliographic pointers):  subword
complexity, arithmetic complexity (related to Van Der Waerden theorem about arithmetic progression $r$-coloring of the integers) \cite{39}, pattern complexity \cite{41}, cyclic complexity, abelian and $k$-abelian complexities \cite{42} or binomial complexity \cite{43}.

Very recently,  prefix palindromic length complexity was considered in \cite{45}.
Furthermore,  let us also mention window complexity \cite{46}.

\medskip
    
    \textbf{Problem 2.}
    \textit{What are the different complexities for our word  $\textsc{\textbf{w}}$, i.e.,
    for the sequence of the most significant digits of $2^{n^d}$?}

\subsection{Significant digit}
We can generalise the problem by considering not the most significant digit but the 
$k^{\text{th}}$
one.

\medskip

Here, we can state the following problem, which is related to this research and was proposed by B. Kra \cite{47}:

\medskip
    
    \textbf{Problem 4.}
    \textit{What is the complexity function of
    the $k^{\text{th}}$ significant digits of $2^{n^d}$?}

\medskip

\newpage

\section*{Acknowledgments}
We would like to greatly thank Alexei  Kanel-Belov who gave us this interesting research problem.
There are no words that can express our gratitude towards our friend,
Michel Rigo for the valuable comments he made.
Next, the first author would like to thanks Julien Cassaigne for his feedback on our problem and wonderful discussion.


\medskip

\normalsize
\begin{thebibliography}{99}



\bibitem{37}
A. Aberkane.
“Words whose complexity satisfies  $\lim \frac{\textsc{p}(n)}{n}=1$.''
{\it
Theor. Comput. Sci.}
{\bf 307} (2003): 31--46.

\url{https://doi.org/10.1016/S0304-3975(03)00091-4}


\bibitem{18}
P.
Alessandri.
“Codages de rotations et basses  complexités.''
Ph.D. Thesis,
Aix-Marseille  University II,
Marseille,
France
(1996).



\bibitem{14}
P.
 Alessandri and
 V. Berthé.
“Three distance theorems and combinatorics on words.''
{\it
Enseign. Math.}
 {\bf 44}
(1998):
103--132.



\bibitem{6}
J.-P. Allouche.
“Surveying some notions of complexity for finite and
infinite sequences.''
{\it
RIMS Kokyuroku Bessatsu.}
{\bf B34}
(2012):
27--37.

\url{http://hdl.handle.net/2433/198090}




\bibitem{39}
S. V. Avgustinovich, D. G. Fon-Der-Flaass, and A. E. Frid.
“ Arithmetical Complexity of Infinite Words.''
{\it Words, Languages \& Combinatorics III.}
(2003): 51--62.
Edited by
M. Ito and
T. Imaoka.
 
\url{https://doi.org/10.1142/9789812704979_0004}





\bibitem{41}
S. V. Avgustinovich, A. Frid, T. Kamae, and P. Salimov. 
“ Infinite permutations of lowest maximal pattern complexity.''
{\it Theor. Comput. Sci.}
 {\bf 412} (2011): 2911--2921.
 
\url{https://doi.org/10.1016/j.tcs.2010.12.062}


\bibitem{31}
A. Y. Belov, G. V. Kondakov, I. V. Mitrofanov,  and M. M. Golafshan. 
“On the
sequence of the first binary digits of the fractional parts of the values of a
polynomial.''
{\it Chebyshevskii Sb.}  {\bf 22}, no. 1 
(2021): 482--487.

\url{https://doi.org/10.22405/2226-8383-2021-22-1-482-487}


\bibitem{24}
F. Benford.
“The Law of Anomalous Numbers.''
{\it
Proc. Am. Philos. Soc.}
{\bf
78},
no.
4
(1938):
551--572.

\url{https://www.jstor.org/stable/984802}


\bibitem{23}
A. Berger and T. P. Hill.
“An Introduction to Benford's Law.''
{\it Princeton University Press.}
USA
(2015).

\url{https://doi.org/10.23943/princeton/9780691163062.003.0001}

\bibitem{35}
J. Cassaigne.
Personal communication (2022).

\url{http://iml.univ-mrs.fr/~cassaign/}


\bibitem{50}
J. Cassaigne.
“
Sequences with grouped factors.''
{\it Developments in Language Theory III (DLT’97).}
(1998):
211--222.




\bibitem{46}
J. Cassaigne,
I.
Kaboré,
and
T.
Tapsoba.
“On a new notion of complexity on infinite
words.''
{\it Acta. Univ. Sapientiae. Matem.}
{\bf 2},
no. {2}
(2010):
 127--136.
 
 \url{https://acta.sapientia.ro/content/docs/on-a-new-notion-of-complexity-on-infinit.pdf}


\bibitem{38}
J. Cassaigne, S. Labbé, and J. Leroy. 
“A set of sequences of complexity $2n + 1$.''
{\it International Conference on Combinatorics on Words,
LNCS.}
{\bf 10432},
Springer-Cham,
(2017):
144–156.
 Edited by
 S. Brlek,
 F. Dolce,
 C. Reutenauer,
 and
 É. Vandomme.
 
 \url{https://doi.org/10.1007/978-3-319-66396-8_14}



\bibitem{34}
J. Ciegler.
“On a theorem of H. Weyl-Numdam.''
{\it Compos. Math.}
{\bf 21},
no.
2
(1969):
151--154.

\url{https://doi.org/10.24033/msmf.32}

\bibitem{26}
P. Diaconis.
“The distribution of leading digits and uniform distribution mod
$1$.''
{\it Ann. Probab.}
 {\bf 5},
no. 1
(1977):
72--81.

\url{https://www.jstor.org/stable/2242803}



\bibitem{21}
G. Didier.
“Characterization of $N$-writings and application to the study of complexity sequences ultimately $n + c^{ste}$.''
{\it Theor. Comput. Sci.}
{\bf 215},
no.
 1 $\&$ 2
(1999):
31--49.

\url{https://doi.org/10.1016/S0304-3975(97)00122-9}


\bibitem{19}
G. Didier.
“Codages de rotations et fractions continues.''
{\it
J. Number Theory.}
{\bf
71
}
(1998):
275--306.

\url{https://doi.org/10.1006/jnth.1998.2246}



\bibitem{32}
M. Einsiedler and T, Ward.
“Ergodic Theory, with a view towards
Number Theory.''
{\it GTM,
Springer.}
{\bf 259}
(2011).

\url{https://doi.org/10.1007/978-0-85729-021-2}



\bibitem{5}
S.
Ferenczi.
“Complexity of sequences and dynamical systems.''
{\it Discrete Math.}
{\bf 206} (1999): 145--154.

\url{https://doi.org/10.1016/S0012-365X(98)00400-2}




\bibitem{45}
A.  Frid.
“ Prefix Palindromic Length of the
Thue-Morse Word.''
{\it J. Integer Seq.}
{\bf 22}
(2019):
19.7.8.

\url{https://cs.uwaterloo.ca/journals/JIS/VOL22/Frid/frid4.html}




\bibitem{29}
M. Golafshan, A. Kanel-Belov, A. Heinis, and G. Potapov.
“On the complexity
function of leading digits of powers of one-digit primes.''
{\it 52nd (AIMC).} Kerman,  Iran, (2021):13--15.

\url{https://doi.org/10.1109/AIMC54250.2021.9656982}


\bibitem{28}
X. He, A. J. Hildebrand, Y. Li, and Y. Zhang.
“Complexity of Leading Digit Sequences.''
{\it Discrete Math. Theor. Comput. Sci.}
{\bf 1},
no. 1
(2020).

\url{https://doi.org/10.23638/DMTCS-22-1-14}







\bibitem{36}
A. Heinis. 
“Arithmetics and combinatorics of words of low complexity.''
Ph.D. Thesis, University of
Leiden, 
Leiden,
Netherlands
(2001).

\url{https://www.math.leidenuniv.nl/~tijdemanr/heinis2.pdf}

\bibitem{30}
A.
Kanel-Belov, G. V. Kondakov,
and
I. V. Mitrofanov.
“Inverse problems of symbolic dynamics.''
{\it Banach Cent. Publ.}
 {\bf 94},
 no. 1
 (2011):
 43--60.
 
\url{DOI: 10.4064/bc94-0-2}

\bibitem{7}
T. Kamae.
“Behavior of various complexity functions.''
{\it
Theor. Comput. Sci.}
{\bf 420}
(2012):
36--47.

\url{https://doi.org/10.1016/j.tcs.2011.11.012}



\bibitem{42}
G. Fici and
S.
Puzynina.
“Abelian combinatorics on words: A survey.''
{\it Comput. Sci. Rev.}
 {\bf 47},
 no.
  100532
 (2023).
 
\url{https://doi.org/10.1016/j.cosrev.2022.100532}



 \bibitem{47}
 B. Kra.
 Personal communication (2022).
 
\url{https://sites.math.northwestern.edu/~kra/}


\bibitem{25}
L.
Kuipers and
H.
Niderreiter.
“Uniform Distribution of Sequences.''
{\it Pure and applied mathematics,
A Wiley-Interscience publication.}
New York
(1974).

\url{https://doi.org/10.1137/1019028}


\bibitem{8}
M. Lothaire. “Algebraic combinatorics on words.''
{\it Encl. of math.
and Its Applications.}
{\bf 90},
 Cambridge University Press (2002).
 
\url{https://doi.org/10.1017/CBO9781107326019}



\bibitem{9}
M. Morse and G. A. Hedlund.
“Symbolic dynamics.''
{\it Amer. J. Math.} 
\textbf{60}
(1938): 815--866.

\url{https://doi.org/10.2307/2371264}




\bibitem{22}
S. Newcomb. 
“Note on the Frequency of Use of Different Digits in Natural Numbers.''
{\it Am. J. Math.}
{\bf 4}
(1881):
39--40.

\url{https://doi.org/10.2307/2369148}


 \bibitem{52}
 M. Queffélec.
 “Substitution
Dynamical Systems--
Spectral Analysis.''
 {\it Lect. Notes Math. 
 Springer.}
{\bf 1294} (2010).

\url{https://doi.org/10.1007/BFb0081890}


\bibitem{43}
M. Rigo and P. Salimov. 
“Another generalization of abelian equivalence: binomial complexity of infinite words.''
{\it Theor. Comput. Sci.}
{\bf 601} (2015): 47--57.

\url{https://doi.org/10.1016/j.tcs.2015.07.025}

\bibitem{20}
G.
Rote.
“Sequences with subword complexity $2n$.''
{\it J. Number Theory.}
{\bf
46
}
(1994):
196--213.

\url{https://doi.org/10.1006/jnth.1994.1012}

\bibitem{33}
H.
Weyl.
“Über die Gleichverteilung von Zahlen mod. Eins.''
{\it Math. Ann.}
{\bf 77}
(1996):
 313--352.
 
\url{https://doi.org/10.1007/BF01475864}

\bibitem{27}
www.problems.ru: 
problem {\bf 98247}.

\url{https://problems.ru/view_problem_details_new.php?id=98247&x=5&y=12}


\end{thebibliography}
\end{document}